\documentclass{elsarticle}

\usepackage{hyperref}

\journal{\empty}

\usepackage{xcolor}
\usepackage{amsfonts,amsmath,amsthm}
\usepackage{graphicx}
\usepackage{epstopdf}
\usepackage{algpseudocode}
\usepackage{amssymb}
\usepackage{amstext}
\usepackage{tabularx}
\usepackage{multirow}
\usepackage[capitalise]{cleveref}
\usepackage{algorithm}

\ifpdf
\DeclareGraphicsExtensions{.eps,.pdf,.png,.jpg}
\else
\DeclareGraphicsExtensions{.eps}
\fi

\usepackage{enumitem}
\setlist[enumerate]{leftmargin=.5in}
\setlist[itemize]{leftmargin=.5in}

\newcommand{\reals}{\mathbb{R}}

\newcommand{\q}{\mathbb{Q}}

\newcommand{\cm}[1]{\operatorname{CM}_{#1}}

\newcommand{\grobner}{Gr\"obner}

\newcommand{\res}[3]{\operatorname{Res}(#1, #2, #3)}

\newcommand{\amat}[1]{\operatorname{\mathcal A}(#1)} 
\newcommand{\smat}[1]{\operatorname{\mathcal S}_{#1}} 
\newcommand{\cres}[3]{\operatorname{CRes}(#1, #2, #3)}

\newtheorem{theorem}{Theorem}
\newdefinition{definition}{Definition}
\newtheorem{lemma}{Lemma}
\newdefinition{mproblem}{Main Problem}
\newdefinition{example}{Example}
\newdefinition{problem}{Open Problem}

\graphicspath{{./Figures/}}

\begin{document}
	
\begin{frontmatter}
	\title{Faster algorithms for circuits in the Cayley-Menger algebraic matroid}
	
	\author{Goran Mali\'c}
	\address{Computer Science Department, Smith College, Northampton, MA, USA}
	\ead{gmalic@smith.edu}
	
	\author{Ileana Streinu\corref{cor1}}
	\address{Computer Science Department, Smith College, Northampton, MA, USA}
	\ead{istreinu@smith.edu}
	
\cortext[cor1]{Corresponding author}

\begin{abstract}
A classical problem in Distance Geometry, with multiple practical applications (in molecular structure determination, sensor network localization etc.) is to find the possible placements of the vertices of a graph with given edge lengths.  For minimally rigid graphs, the double-exponential \grobner \ Bases algorithm with an elimination order can be applied, in theory, but it is impractical even for small instances. By relating the problem to the computation of circuit polynomials in the Cayley-Menger ideal, \cite{malic:streinu:CombRes:socg:2021} recently proposed an algebraic-combinatorial approach and an elimination algorithm for circuit polynomials. It is guided by a tree structure whose leaves correspond to complete $K_4$ graphs and whose nodes perform algebraic resultant operations. 
  
In this paper we uncover further combinatorial structure in the Cayley-Menger algebraic matroid that leads to an extension of their algorithm. In particular, we generalize the combinatorial resultant operation of \cite{malic:streinu:CombRes:socg:2021} to take advantage of the non-circuit generators and irreducible polynomials in the Cayley-Menger ideal and use them as leaves of the tree guiding the elimination. Our new method has been implemented in Mathematica and allows previously un-obtainable calculations to be carried out. In particular, the $K_{3,3}$-plus-one circuit polynomial, with over one million terms in $10$ variables and whose calculation crashed after several days with the previous method of \cite{malic:streinu:CombRes:socg:2021}, succeeded now in approx. 30 minutes.
\end{abstract}

\begin{keyword}
  Cayley-Menger ideal, rigidity matroid, circuit polynomial, combinatorial resultant, inductive construction, Gr\"obner basis elimination
\end{keyword}
	
\end{frontmatter}

\section{Introduction.}
\label{app:sec:introduction}

Given a graph $G=(V,E)$ and a collection of edge weights $\{ w_e | e\in E \}$, the main problem in Distance Geometry \cite{crippen:havel:distanceGeometry:1988,lavor:liberti:distanceGeometry:2017,alfakih:distanceMatrices:2018} asks for \emph{placements} of the vertices $V$ in some Euclidean space (in this paper, 2D) in such a way that the resulting edge lengths match the given weights. The problem has many applications, in particular in molecular structure determination \cite{emiris:mourrain:molConformations:1999,mucherino:lavor:liberti:maculan:distanceGeometry:2013} and sensor network localization
\cite{so:ye:theorySemidefinite:2005}.
 A polynomial-time related problem is to find the possible values of a \emph{single unknown distance} corresponding to a \emph{non-edge} (a pair of vertices that are not connected by an edge in $G$). A system of quadratic equations can be easily set up so that the possible values of the unknown distance are the (real) solutions of this system. In theory, the general but impractical (double-exponential time) Gr\"obner basis algorithm with an elimination order can be used to solve the system. The set of solutions can be finite (if the given weighted graph is \emph{rigid}) or continuous (if it is \emph{flexible}). 

By formulating the \emph{Single Unknown Distance Problem} in Cayley coordinates (with variables $x_{ij}$ corresponding to edges $ij\in E$),  the problem can be reduced to finding a specific irreducible polynomial in the Cayley-Menger ideal, called the  \emph{circuit polynomial}. The unknown distance is then computed by solving the uni-variate polynomial obtained by substituting the given edge lengths in the circuit polynomial.

The topic of this paper is the effective computation of \emph{circuit polynomials in the 2D Cayley-Menger ideal.} These are multi-variate polynomials in the ideal generated by the rank conditions of the Cayley matrix associated to a generic set of $n$ points. The supports of such polynomials correspond to a family of sparse graphs called \emph{rigidity circuits}, i.e. \emph{circuits in the (combinatorial) rigidity matroid}. They are particularly difficult to compute, as the number of terms in a circuit polynomial grows very rapidly as a function of the number of nodes in the underlying graph.

\subparagraph{Background and related work.} The literature on distance geometry is vast, with classical results   \cite{blumenthal:distanceGeometry:1953,crippen:havel:distanceGeometry:1988} (some going back to Cayley in the 19th century), recent applications \cite{streinu:borcea:numberEmbeddings:2004,mucherino:lavor:liberti:maculan:distanceGeometry:2013,lavor:liberti:distanceGeometry:2017} and new techniques emerging from various areas: theory of semidefinite programming, distance matrices, rigidity theory, combinatorics and computational algebra \cite{alfakih:distanceMatrices:2018,so:ye:theorySemidefinite:2005,sitharam:convexConfigSpaces:2010,emiris:tsigaridas:varvitsiotis:mixedVolume:2013,capco:schicho:realizations:2017,bartzos:emiris:etAl:realizations:2021}. Our paper relies on insights from the theory of matroids \cite{Oxley:2011,VDWmoderne}, as well as from the distinguished properties of circuit polynomials in algebraic matroids \cite{dressLovasz}. A computational perspective to circuit polynomials, recently popularized by \cite{rosen:sidman:theran:algebraicMatroidsAction:2020}, emerged from \cite{rosen:thesis}. There, circuits in algebraic matroids of \emph{arbitrary} polynomial ideals were studied and  small cases were explored using Macaulay2 code \cite{rosen:GitHubRepo}. They also pointed out the smallest circuit polynomial example in the context of 2D distance geometry. This polynomial was the only known example until last year: it is supported by the $K_4$ graph and, being a generator of the Cayley-Menger ideal, needs no calculation. 

Computing other circuit polynomials would require the double-exponential time \grobner \ Bases algorithm with an elimination order. The longest successful calculation carried out in this fashion (5 days and 6 hrs, using Mathematica's GroebnerBases function and some combinatorial considerations) was reported in \cite{malic:streinu:CombRes:socg:2021} for the Desargues-plus-one circuit graph from \cref{fig:desarguesAndK33}(left). 

The starting point of our paper is a recent result of \cite{malic:streinu:CombRes:socg:2021}, who proposed a mixed combinatorial-algebraic approach that led to faster calculations of non-trivial examples. This algorithm is guided by a tree structure whose leaves correspond to complete $K_4$ graphs and whose nodes correspond to rigidity circuits obtained by a combinatorial operation abstracted from the classical algebraic (Sylvester) resultant. This method, best described in the language of algebraic matroids, uncovers a rich algebraic-combinatorial structure in the Cayley-Menger ideal. The resulting algorithm significantly outperforms the general method: the previous example that took over 5 days to compute with GroebnerBases was completed in less than 15 seconds. 

\begin{figure}[ht]
	\centering
	\includegraphics[width=0.24\textwidth]{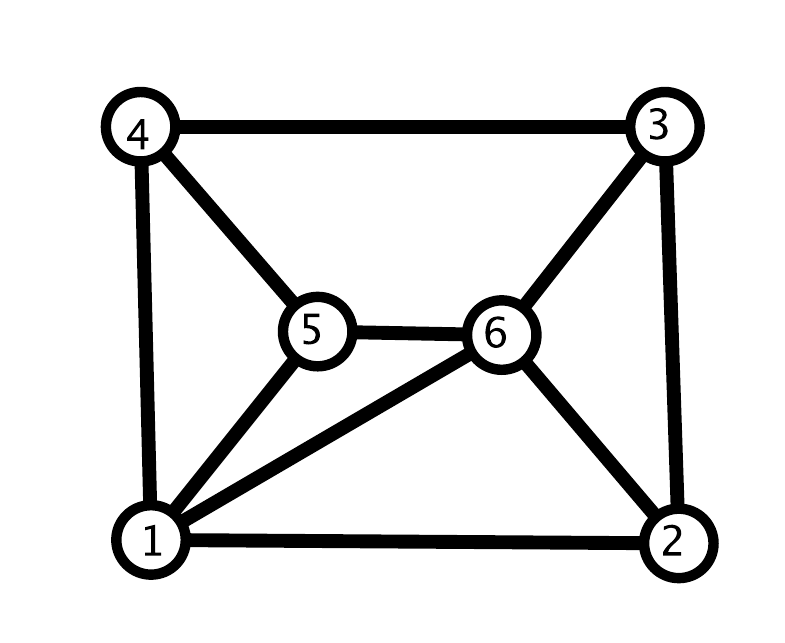}
	\includegraphics[width=0.24\textwidth]{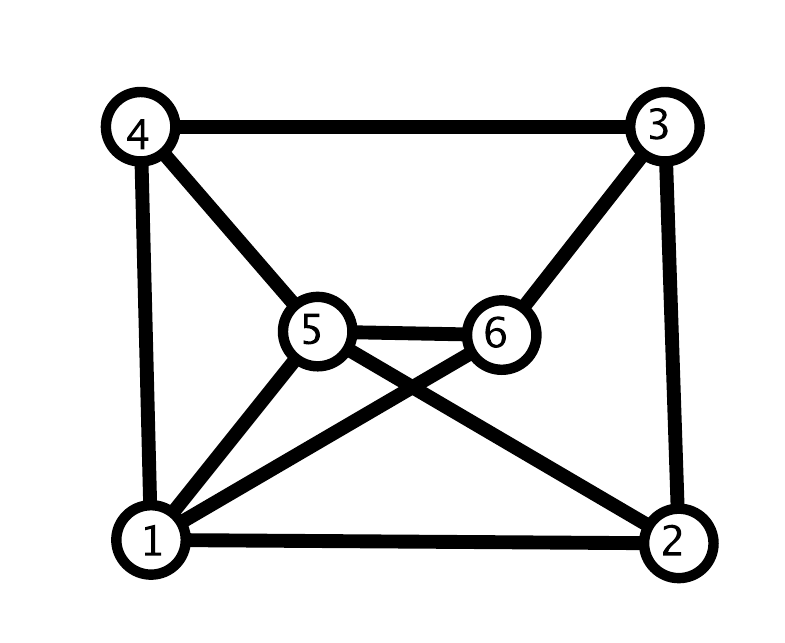}
	\vspace{-12pt}
	\caption{The Desargues-plus-one and the $K_{3,3}$-plus-one rigidity circuits.}
	\label{fig:desarguesAndK33}
\end{figure}

\subparagraph{Our Results.} 
In this paper we generalize the \emph{combinatorial resultant tree} introduced in \cite{malic:streinu:CombRes:socg:2021} from circuits to dependent sets in the rigidity matroid. This allows us to extend the algorithm of \cite{malic:streinu:CombRes:socg:2021} to take advantage of the other non-circuit generators of the Cayley-Menger ideal besides those supported on $K_4$ graphs. We show that these extended trees can lead to more tractable algebraic calculations than those based exclusively on circuits. Implemented in Mathematica, executions that were prohibitively expensive using the previous method now ran to completion. In particular, the polynomial for the $K_{3,3}$-plus-one circuit (\cref{fig:desarguesAndK33}, right), with over one million terms in $10$ variables and whose calculation crashed after several days with the method of \cite{malic:streinu:CombRes:socg:2021}, succeeded now in approx. 30 minutes.
 
\subparagraph{Overview of the paper.} In \cref{sec:prelimCombinatorial} we briefly introduce the necessary discrete structures from 2D rigidity theory, matroids and combinatorial resultants. The extended \emph{combinatorial resultant tree} is defined and characterized in \cref{sec:extendedCombRes}. A brief introduction to algebraic matroids, elimination ideals and resultants, the Cayley-Menger ideal and its circuit polynomials appears in \cref{sec:prelimAlgebraic}. We discuss the dependent, non-circuit generators of the Cayley-Menger ideal in \cref{sec:generatorsCM}. To help build the intuition on how our algorithm works, we demonstrate in \cref{sec:computingK33} how we succeded in computing the $K_{3,3}$-plus-one circuit polynomial. The relevant step of the extended algorithm for computing circuit polynomials is described in \cref{sec:algebraicNew}. We conclude with a few remaining open questions  in \cref{sec:concludingRemarks}.

\section{Preliminaries: rigid graphs and matroids.}
\label{sec:prelimCombinatorial}

In this section we introduce the essential, previously known concepts and results 
from combinatorial rigidity theory in 2D and from \cite{malic:streinu:CombRes:socg:2021} that are relevant for setting up the foundation of our paper and carrying out the proofs. We work with (sub)graphs $G$ given by subsets $E$ of edges of the complete graph $K_n$ on vertices $[n]:=\{1,\cdots,n\}$. We denote with $V(G)$, resp. $E(G)$ the vertex, resp. edge set (support) of $G$. Let $V(E)$ be the \emph{vertex span} of edges $E$, i.e. the set of all vertices appearing in $E$. A subgraph $G$ is \emph{spanning} if its edge set $E(G)$ spans $[n]$. The \emph{neighbours} $N(v)$ of a vertex $v$ are the vertices connected to $v$ by an edge in $G$.

\subparagraph{Frameworks.} A \emph{2D bar-and-joint framework} is a pair $(G,p)$ of a graph $G=(V,E)$ and a planar point set $p$. Its vertices $V=[n]$ are mapped to points $p=\{ p_1,\dots,p_n\}$ in $\mathbb R^2$ via the \emph{placement map} $p\colon V\to\mathbb R^2$ given by $i\mapsto p_i$. We view the edges as \emph{rigid bars} and the vertices as \emph{rotatable joints} which allow the framework to deform continuously as long as the bars retain their original lengths. The \emph{realization space} of the framework is the set of all its possible placements in the plane with the same bar lengths. Two realizations are congruent if they are related by a planar isometry. The \emph{configuration space} of the framework is made of congruence classes of realizations. The \emph{deformation space of a given framework $(G,p)$} is the  connected component of the configuration space that contains this particular placement (given by $p$). A framework is \emph{rigid} if its deformation space consists in exactly one configuration, and \emph{flexible} otherwise.

Combinatorial rigidity theory of bar-and-joint frameworks seeks to understand the rigidity and flexibility of frameworks in terms of their underlying graphs.  The following theorem  \cite{laman:rigidity:1970} relates the rigidity of 2D bar-and-joint frameworks to a specific type of graph sparsity: the hereditary property (b) below is also referred to as the \emph{$(2,3)$-sparsity condition}. 

\begin{theorem}
	\label{thm:laman}
	A bar-and-joint framework $(G,p)$ is \emph{generically} minimally rigid in 2D iff its underlying graph $G=(V,E)$ satisfies two conditions: (a) it has exactly $|E|=2|V|-3$ edges, and (b) any proper subset $V'\subset V$ of vertices spans at most $2|V'|-3$ edges.
\end{theorem}

The \emph{genericity} condition appearing in the statement of this theorem refers to the vertex placement $p$. Without going into details, we retain its most important consequence, namely that \emph{small perturbations of $p$ do not change the rigidity or flexibility properties of a generic framework.} This theorem allows for the rigidity and flexibility of generic frameworks to be studied in terms of only their underlying graphs. 

\begin{figure}[ht]
	\centering
		\includegraphics[width=.24\textwidth]{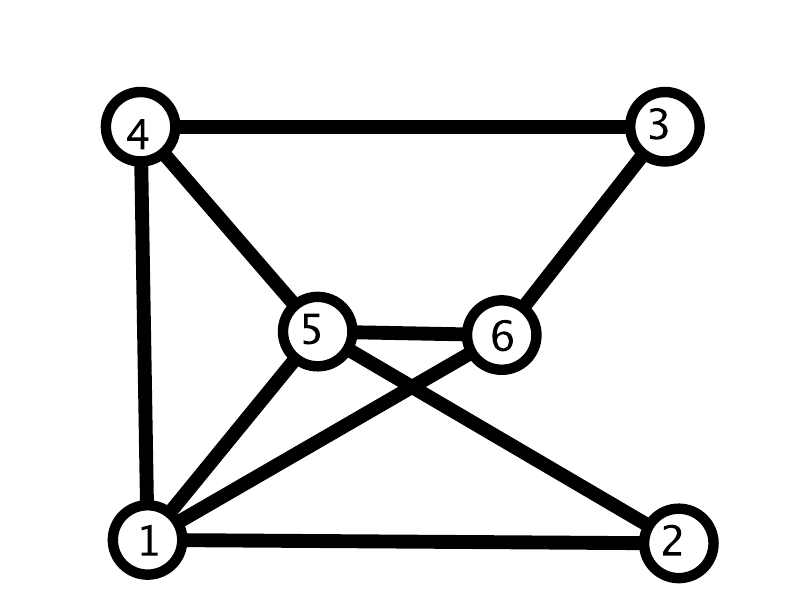}
		\includegraphics[width=.24\textwidth]{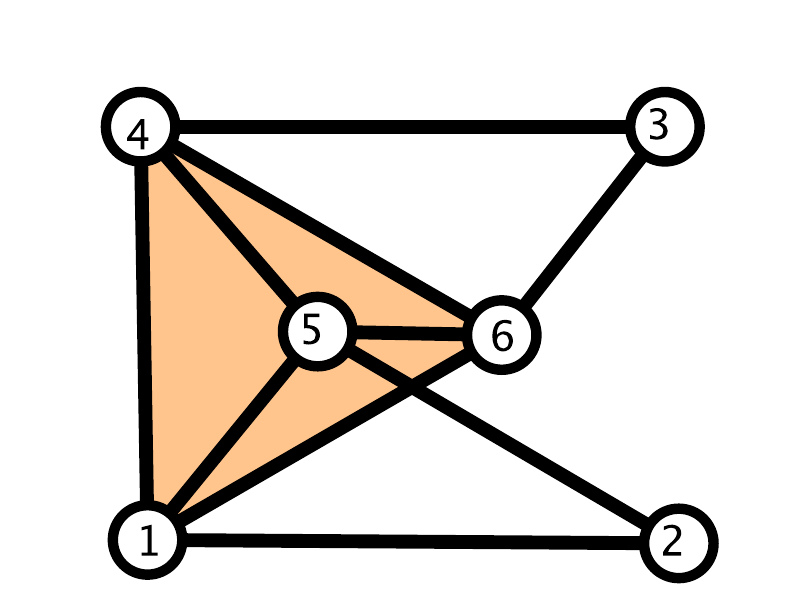}
		\includegraphics[width=.24\textwidth]{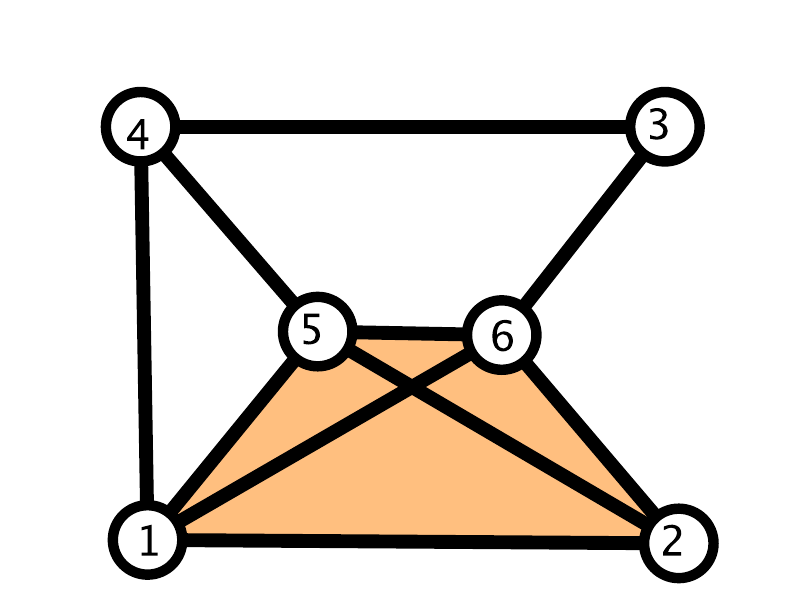}
		\includegraphics[width=.24\textwidth]{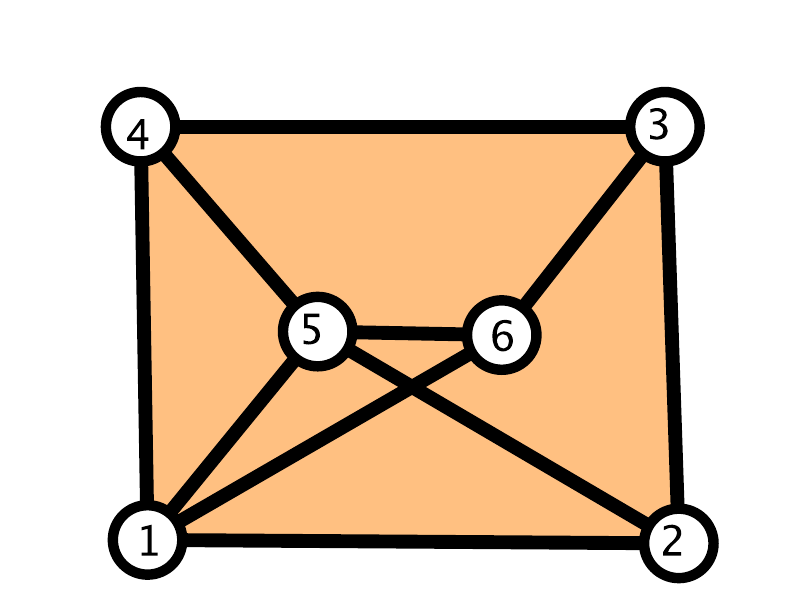}
		\vspace{-4pt}
	\caption{A Laman (maximally independent) graph (left), and three dependent (overconstrained) Laman-plus-one graphs obtained by adding a single edge to it. Each contains a unique circuit (highlighted), and only the rightmost one is a spanning circuit. In the two middle dependent graphs, the edges spanned by the circuit are redundant, and the others are critical.
	}
	\label{fig:lamanAndCircuits}
\end{figure}
A graph satisfying the conditions of this theorem is called a \emph{Laman graph}. It is \emph{minimally rigid} in the sense that it has just enough edges to be rigid: if one edge is removed, it becomes \emph{flexible.} Adding extra edges to a Laman graph keeps it rigid, but the minimality is lost: these graphs are said to be rigid and \emph{overconstrained}. In short, for a graph to be rigid, its vertex set must span a Laman graph; otherwise the graph is flexible. A flexible graph decomposes into edge-disjoint rigid components, each component spanning on its vertex set a Laman graph with (possibly) additional edges. \cref{fig:lamanAndCircuits} and \cref{fig:redundantCritical} illustrate the concepts. These properties have matroidal interpretations.

\begin{figure}[ht]
	\centering
		\includegraphics[width=.24\textwidth]{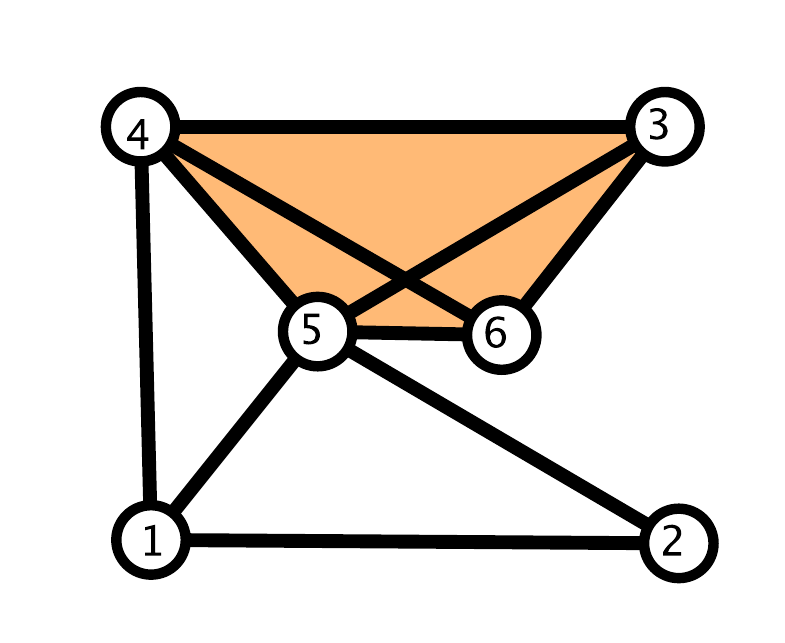}
		\includegraphics[width=.24\textwidth]{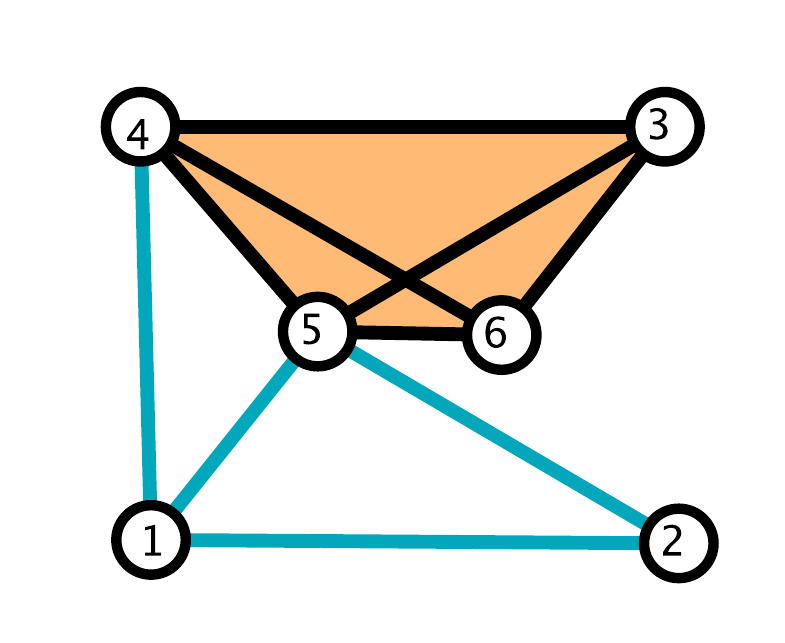}
		\includegraphics[width=.24\textwidth]{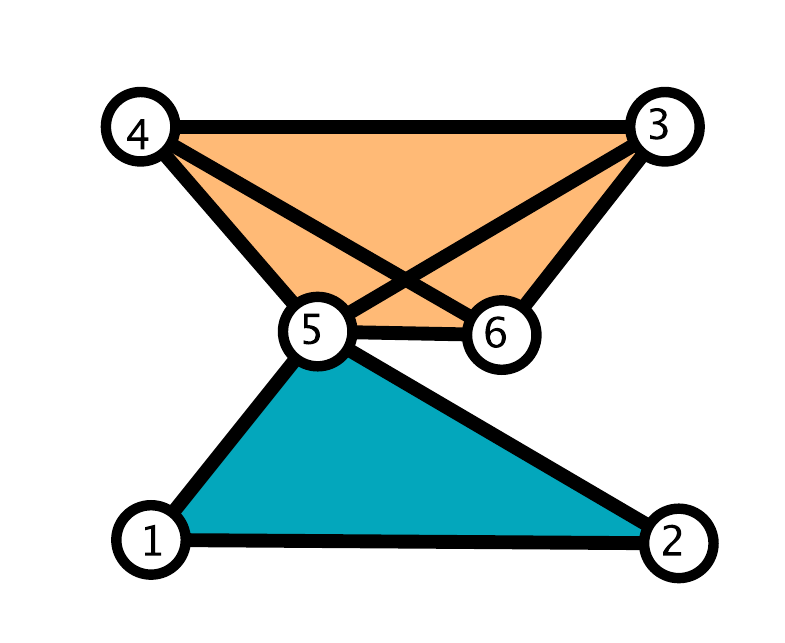}
		\vspace{-4pt}
	\caption{(Left) A dependent graph containing a circuit (highlighted). (Middle) Its redundant edges are in black and its critical edges are colored green. (Right) The removal of a critical edge creates a flexible graph, whose two edge-disjoint rigid components are highlighted: orange (dependent) and green (independent). 
	}
	\vspace{-10pt}
	\label{fig:redundantCritical}
\end{figure}

\subparagraph{Matroids. }  A matroid \cite{Oxley:2011} is an abstraction capturing (in)dependence relations among collections of elements from a \emph{ground set}, and is inspired by both \emph{linear} dependencies (among, say, rows of a matrix) and by \emph{algebraic} constraints imposed by algebraic equations on a collection of otherwise free variables. The standard way to specify a matroid is via its \emph{independent sets}, which have to satisfy certain axioms \cite{Oxley:2011} (skipped here, since they are not relevant for our presentation). A \emph{base} is a maximal independent set and a set which is not independent is said to be \emph{dependent}. A minimal dependent set is called a \emph{circuit}. Relevant for our purposes are the following general aspects: (a) (hereditary property) a subset of an independent set is also independent; (b) all bases have the same cardinality, called the \emph{rank} of the matroid. Further properties will be introduced in context, as needed.

In this paper we encounter two types of matroids: a \emph{graphic rigidity matroid}, defined on a ground set given by all the edges $E_n:=\{ij: 1\leq i < j \leq n\}$ of the complete graph $K_n$; this is the \emph{$(2,3)$-sparsity matroid} or the \emph{generic rigidity matroid} described below; 
and an \emph{algebraic matroid},  defined on an isomorphic ground set of variables $X_n:=\{x_{ij}: 1\leq i < j \leq n\}$; this is the \emph{algebraic rigidity matroid associated to the Cayley-Menger ideal} that will be defined in \cref{sec:prelimAlgebraic}. Two examples of (spanning) circuits in the (graphic) rigidity matroid on $6$ vertices were given in \cref{fig:desarguesAndK33}.

\subparagraph{The $(2,3)$-sparsity matroid: independent sets, bases, circuits.}
The $(2,3)$-sparse graphs on $n$ vertices form the collection of independent sets for a matroid $\smat{n}$ on the ground set $E$ of edges of the complete graph $K_n$, called the (generic) \emph{2D rigidity matroid}, or the \emph{$(2,3)$-sparsity matroid} \cite{graverServatiusServatius}. The bases of the matroid $\smat n$ are the maximal independent sets, hence the Laman graphs. A set of edges which is not sparse is a \emph{dependent} set. A \emph{minimal dependent set} is a (sparsity) \emph{circuit}. The edges of a circuit span a subset of the vertices of $V$. For instance, adding one edge to a Laman graph creates a dependent set of $2n-2$ edges, called a Laman-plus-one graph (\cref{fig:lamanAndCircuits}), which contains a unique circuit. A circuit spanning $V$ is said to be a \emph{maximal} or \emph{spanning} circuit in the sparsity matroid $\smat n$; it has $2n-2$ edges on $n$ vertices and satisfies the Laman $(2,3)$-sparsity property on all strict subsets of $n'<n$ vertices. Examples are given in \cref{fig:lamanAndCircuits}.



A \emph{dependent rigid graph} $G$ is a dependent graph which contains a spanning Laman graph; for example, the three rightmost graphs in \cref{fig:lamanAndCircuits}. An edge $e\in G$ is said to be \emph{redundant} if the graph $G\setminus \{ e \}$ is still rigid, otherwise the edge is said to be \emph{critical}: its removal makes the graph \emph{flexible}. It is well known that (a) all the edges of a circuit are redundant and that (b) a redundant edge is contained in a subgraph that spans a circuit (see e.g. \cite{jackson:jordan:connectedRigMatroids:2005}). Examples of dependent graphs with critical edges are those Laman-plus-one graphs which are not themselves spanning circuits: each one contains a unique, strict subgraph which is a circuit, and all the edges that are not in this circuit are critical (\cref{fig:lamanAndCircuits}). We'll make use of these concepts in \cref{sec:extendedCombRes}.

\subparagraph{Combining graphs and circuits: 2-sums and combinatorial resultants.} We define now operations that combine two graphs (with some common vertices and edges) into one. If $G_1$ and $G_2$ are two graphs, we use a consistent notation for their number of vertices and edges $n_i=|V(G_i)|$, $m_i=|E(G_i)|$, $i=1,2$, and for their union and intersection of vertices and edges, as in $V_{\cup}=V(G_1)\cup V(G_2)$, $V_{\cap}=V(G_1)\cap V(G_2)$, $n_{\cup}=|V_{\cup}|$, $n_{\cap}=|V_{\cap}|$ and similarly for edges, with $m_{\cup}=|E_{\cup}|$ and $m_{\cap}=|E_{\cap}|$. The \emph{common subgraph} of two graphs $G_1$ and $G_2$ is $G_{\cap} = (V_{\cap}, E_{\cap})$.

Let $G_1$ and $G_2$ be two graphs having exactly two vertices $u, v\in V_{\cap} $ and one edge $uv \in E_{\cap}$ in common. Their {\bf $2$-sum} is the graph $G=(V,E)$ with $V=V_{\cup}$ and $E=E_{\cup} \setminus \{uv\}$. The inverse operation of splitting $G$ into $G_1$ and $G_2$ is called a $2$-split or $2$-separation. It is known \cite{BergJordan} that the $2$-sum of two circuits is a circuit, and the 2-split of a circuit is a pair of circuits. This operation, used for the Tutte decomposition of a $2$-connected graph into $3$-connected components, allows us to focus (when computing circuits and circuit polynomials) on the more challenging case of $3$-connected graphs.

Let $G_1$ and $G_2$ be two distinct graphs with non-empty intersection $E_{\cap} \neq\emptyset$ and let $e\in E_{\cap}$ be a common edge. The {\bf combinatorial resultant} \cite{malic:streinu:CombRes:socg:2021} of $G_1$ and $G_2$  on the \emph{elimination edge} $e$ is the graph $\cres{G_1}{G_2}{e}$ with vertex set $V_{\cup}$ and edge set $E_{\cup}\setminus\{e\}$. The $2$-sum is the simplest kind of combinatorial resultant.

\begin{figure}[ht]
	\centering
	\includegraphics[width=0.4\textwidth]{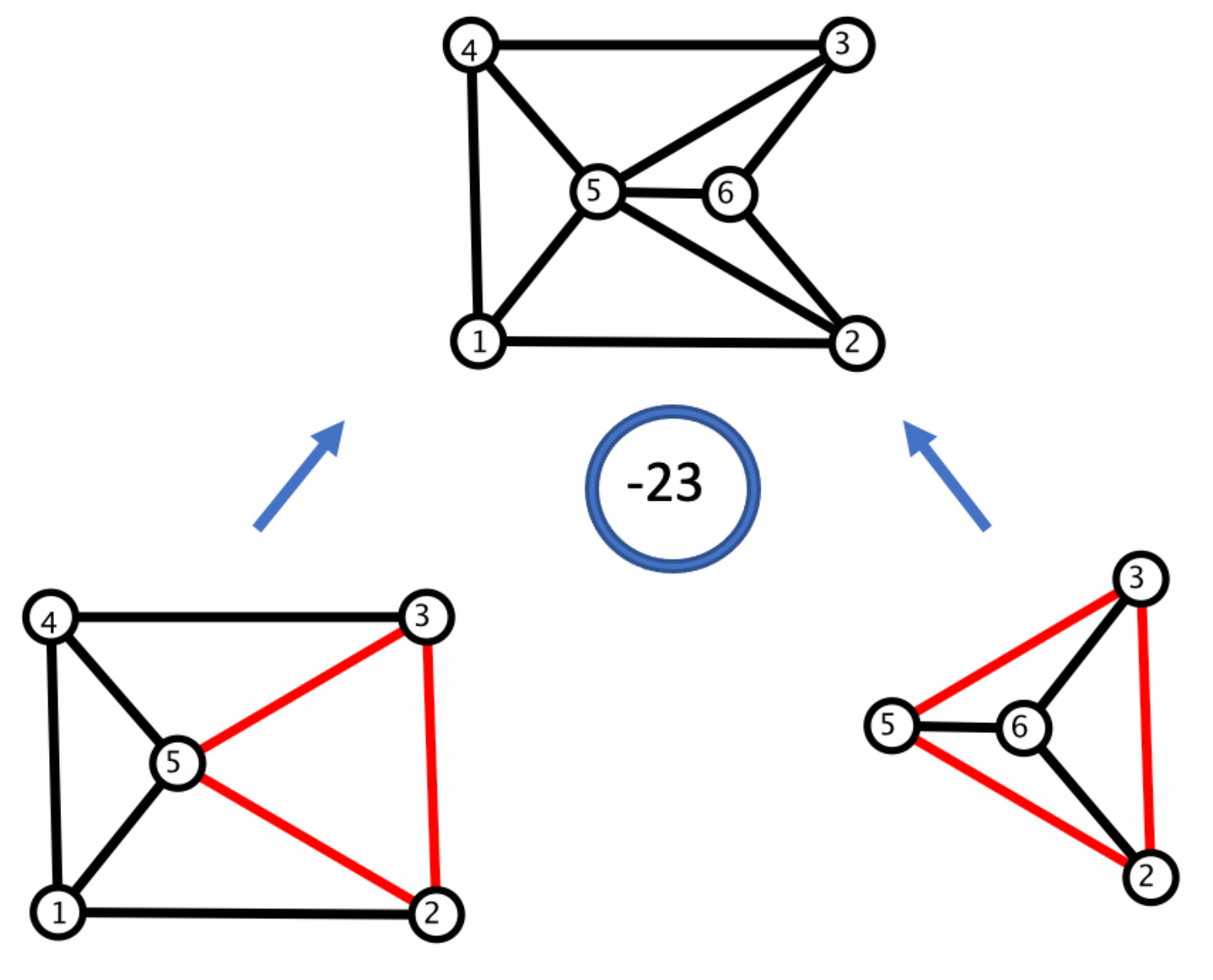}
	\caption{(Bottom) A $4$-wheel $W_4$ and a complete $K_4$ graph, with their common  Laman graph (a triangle) shown in red, the elimination edge (in a circle) and (Top) their combinatorial resultant, the $5$-wheel $W_5$ circuit.}
	\vspace{-10pt}
	\label{fig:combResultant}
\end{figure}

\noindent
\subparagraph{Circuit Resultant Trees. } In \cite{malic:streinu:CombRes:socg:2021}, the focus was on combinatorial resultants that produce circuits from circuits. It was proven there that the combinatorial resultant of two circuits has $m=2n-2$ edges iff the common subgraph $G_{\cap}$ of the two circuits is Laman, in which case the two circuits are said to be \emph{properly intersecting}. If the combinatorial resultant operation applied to two properly intersecting circuits results in a spanning circuit, we say that it is \emph{circuit-valid} (as in \cref{fig:combResultant}). It was further shown in \cite{malic:streinu:CombRes:socg:2021} that each circuit can be obtained via circuit-valid combinatorial resultant operations from a collection of $K_4$ circuits on subsets of $4$ vertices whose union is the span of the desired circuit. This construction (illustrated in \cref{fig:resTree} for the $K{3,3}$-plus-one circuit) is captured by a tree (technically, an \emph{algebraic expression tree}, where the operator at each internal node is a combinatorial resultant). It has the following properties: (a) all its nodes are labeled by circuits, with $K_4$ graphs at the leaves; (b) each internal node corresponds to a combinatorial resultant operation that combines the circuits of the node's two children and eliminates the edge marked under the node. An additional property of the tree produced by the method of \cite{malic:streinu:CombRes:socg:2021} is that, when combining two circuits with the resultant operation, it always adds at least one new vertex.  Hence the depth of the tree is at most $n-4$; it is exactly $2$ in \cref{fig:resTree} for a circuit with $n=6$ vertices. 

This type of tree was called in \cite{malic:streinu:CombRes:socg:2021} a \emph{combinatorial} resultant tree. For a more descriptive terminology, in this paper we refer to this object with an additional qualifier as a \emph{circuit (combinatorial) resultant tree}, reserving the \emph{combinatorial resultant tree} for the more general structure defined in the next section. 






\begin{figure}[ht]
	\centering
		\includegraphics[width=.8\textwidth]{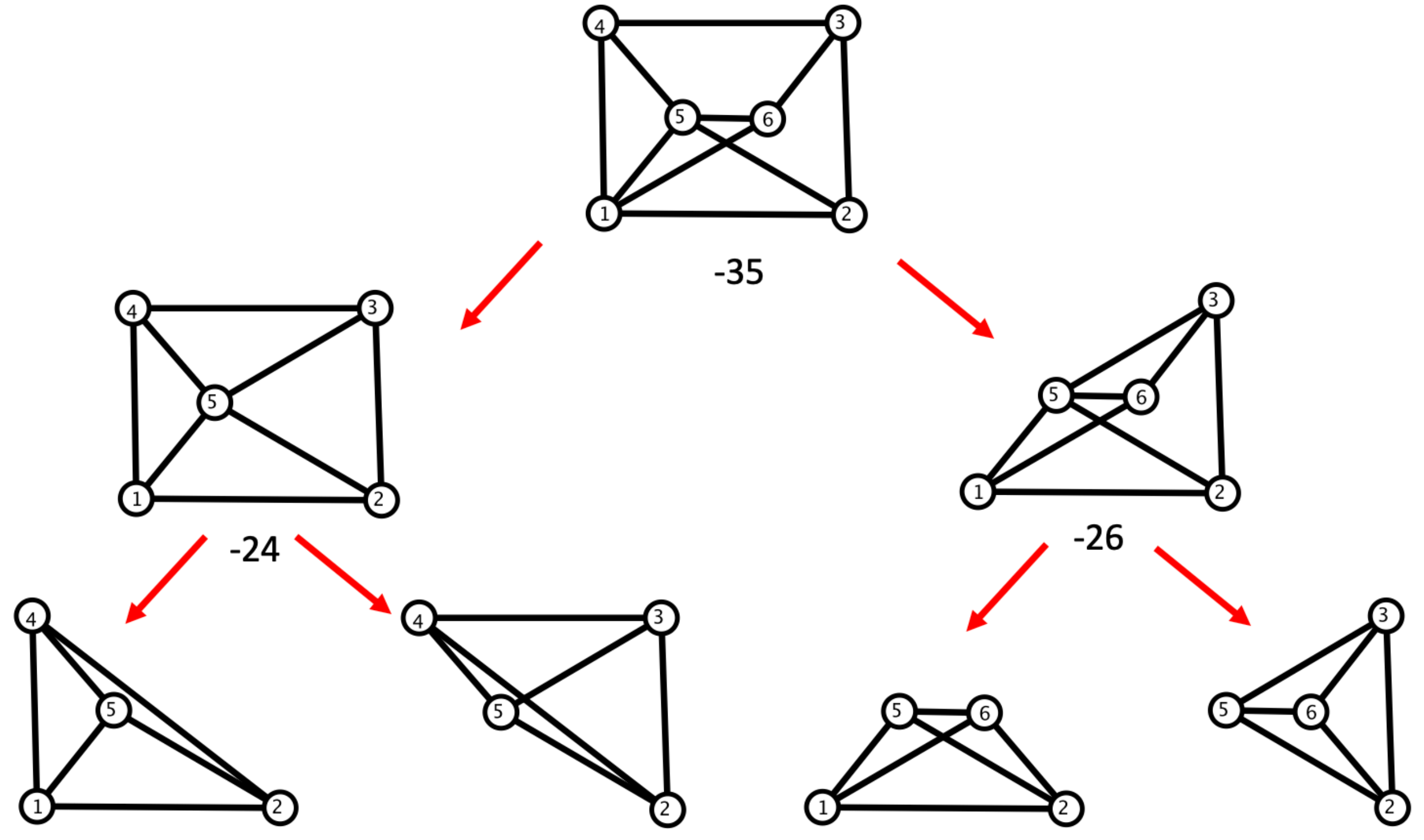}
	\caption{
	A combinatorial circuit resultant tree for $K_{3,3}$-plus-one: all of its nodes are circuits. }
	\label{fig:resTree}
\end{figure}

\section{Combinatorial Resultant Trees.}
\label{sec:extendedCombRes}

We relax now some of the constraints imposed on the resultant tree by the construction from \cite{malic:streinu:CombRes:socg:2021}. The internal nodes correspond, as before, to combinatorial resultant operations, but: (a) they are no longer restricted to be applied only on circuits or to produce only circuits; (b) the leaves can be labeled by graphs other than $K_4$'s, and (c) the sequence of graphs on the nodes along a path from a leaf to the root is no longer restricted to be strictly monotonely increasing in terms of the graphs' vertex sets. 

\begin{definition}
	\label{def:generators}
	A finite collection $Gen$ of\emph{dependent graphs} such that $K_4 \in Gen$ will be called a \emph{set of generators}. 
\end{definition}

The generators will be the graphs allowed to label the leaves. For the purpose of generating (combinatorial) circuits and computing (algebraic) circuit polynomials, meaningful sets of generators will be discussed in \cref{sec:generatorsCM}. We restrict the generators to  graphs which are \emph{dependent} in the rigidity matroid because these correspond precisely to the supports of polynomials in the Cayley-Menger ideal, as we will see in \cref{sec:prelimAlgebraic,sec:generatorsCM}.

\begin{definition}
	\label{def:resultantTree}
	A \emph{combinatorial resultant tree} with generators in $Gen$ is a finite binary tree such that: (a) its leaves are labeled with graphs from $Gen$, and (b) each internal node $i$ (marked with a graph $G_i$ and an edge $e_i\not\in G_i$) corresponds to a combinatorial resultant operation applied on the two graphs $G_1, G_2$ labeling its children. Specifically,  $G_i=\cres{G_1}{G_2}{e_i}$, where the edge $e_i\in G_1\cap G_2$.
\end{definition}

An example is illustrated in \cref{fig:resTreeK33}.

\begin{figure}[ht]
	\centering
		\includegraphics[width=.8\textwidth]{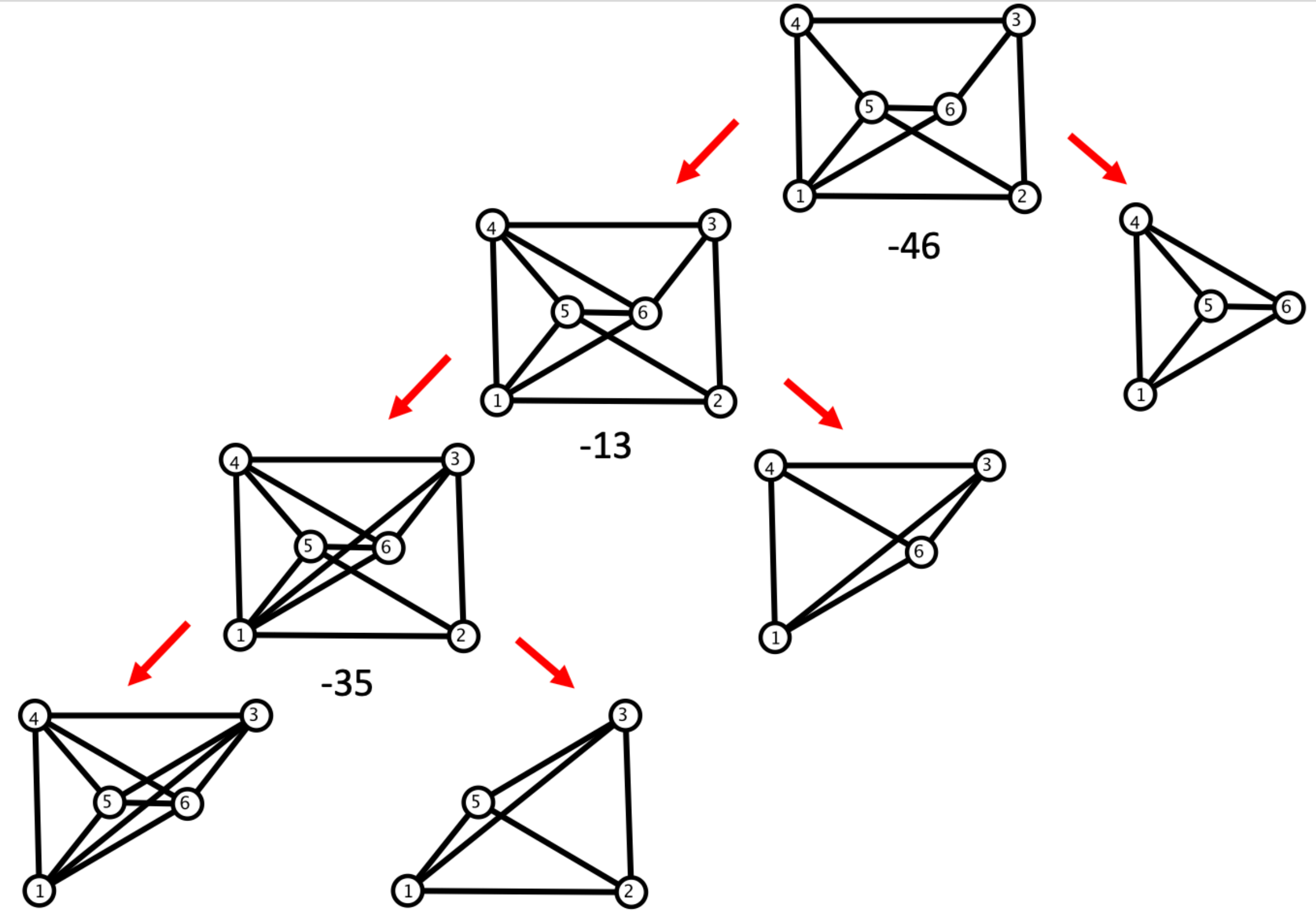}
	\caption{
	A (generalized) \emph{combinatorial resultant tree} for the $K_{3,3}$-plus-one circuit: its leftmost leaf and the two internal nodes along the leftmost path to the root are labeled with rigid dependent graphs which are \emph{not} circuits. 
	}
	\label{fig:resTreeK33}
\end{figure}

\begin{lemma}
	\label{lem:dependent}
	If the generators \emph{Gen} are dependent graphs (in the rigidity matroid), then all the graphs labeling the nodes (internal, not just the leaves) of a combinatorial resultant tree are also \emph{dependent}.
\end{lemma}

\begin{proof}

The proof is an induction on the tree nodes, with the base cases at the root. For the inductive step,  assume that $G_1$ and $G_2$ are the dependent graphs labeling the two children of a node labeled with $G=CRes(G_1,G_2,e)$, where $e\in E_{\cap}$ is an edge in the common intersection $G_{\cap}$. We consider two cases, depending on whether $e$ is critical in none or in at least one of $G_1$ and $G_2$. In each case, we identify a subset of the combinatorial resultant graph $G$ which violates Laman's property, hence we'll conclude that the entire graph $G$ is dependent.
	
{\bf Case 1: $e$ is redundant in both $G_1$ and $G_2$.} This means that there exist subsets of edges $C_1\subset G_1$ and $C_2\subset G_2$, both containing the edge $e$, which are circuits (their individual spanned-vertex sets may possibly contain additional edges, but this only makes it easier to reach our desired conclusion). Their intersection $C_1\cap C_2$ cannot be dependent (by the minimality of circuits). Hence their union, with edge $e$ eliminated, has at least $2n_{\cup}-2$ edges (cf. the proof of Lemma 3.1. in \cite{malic:streinu:CombRes:socg:2021}), hence it is dependent. 

{\bf Case 2: $e$ is critical either in $G_1$ or in $G_2$.} Let's assume it is critical in $G_1$. Since $G_1$ is dependent and $e\in G_1$ is critical, it means that the removal of $e$ from $G_1$ creates a flexible graph which is still dependent. As a flexible graph, it splits into edge-disjoint rigid components; in this case, at least one of these components $R$ is dependent. Then, since the removal of $e$ does not affect $R$,  $R$ remains dependent in the resultant graph $G=CRes(G_1,G_2,e)$.
\end{proof}

\begin{definition}
	\label{def:rdepResultantTree}
	Given a circuit $C$, a\emph{valid combinatorial resultant tree} for $C$ is a combinatorial resultant tree with root $C$ and whose leaves (and hence nodes) are dependent graphs.
\end{definition}

The example in \cref{fig:resTreeK33} is a valid combinatorial resultant tree for the $K_{3,3}$-plus-one circuit. After reviewing the necessary algebraic notions in the next section, we will use it in \cref{sec:computingK33} to demonstrate our extended algebraic elimination algorithm, which will be then described in \cref{sec:algebraicNew}.

\section{Preliminaries: The Cayley-Menger Ideal and Algebraic Matroid.}
\label{sec:prelimAlgebraic}

We turn now to the algebraic aspects of our problem. For full details, the reader may consult the comprehensive sections 2-6 of \cite{malic:streinu:CombRes:arxiv:2021}.



\subparagraph{Polynomial ideals.}  Let $X$ be an arbitrary set of variables. A set of polynomials $I \subset\q[X]$ is an \emph{ideal of} $\q[X]$ if it is closed under addition and multiplication by elements of $\q[X]$. 
A \emph{generating set} for an ideal is a set $S\subset \q[X]$ of polynomials such that every polynomial in the ideal is an algebraic combination of elements in $S$ with coefficients in $\q[X]$. 
Ideals generated by a single polynomial are called \emph{principal}. An ideal $I$ is said to be \emph{prime} if, whenever $fg\in I$, then either $f\in I$ or $g\in I$. A polynomial is \emph{irreducible} (over $\q$) if it cannot be decomposed into a product of non-constant polynomials in $\q[X]$. A principal ideal is prime iff it is generated by an irreducible polynomial. However, an ideal generated by two or more irreducible polynomials is not necessarily prime.

\subparagraph{Elimination ideals.} If $I$ is an ideal of $\mathbb Q[X]$ and $X'\subset X$ non-empty, then the \emph{elimination ideal} of $I$ with respect to $X'$ is the ideal $I\cap \mathbb Q[X']$ of the ring $\mathbb Q[X']$.
Elimination ideals frequently appear in the context of \grobner{} bases \cite{Buchberger, CoxLittleOshea} which give a general approach for computing elimination ideals: if $\mathcal G$ is a \grobner{} basis for $I$ with respect to an \emph{elimination order}, 
 e.g.\ 
 $x_{i_1}>x_{i_2}>\dots>x_{i_n}$, then the elimination ideal $I\cap \q[x_{i_{k+1}},\dots,x_{i_n}]$ which eliminates the first $k$ indeterminates from $I$ in the specified order has $\mathcal G\cap \q[x_{i_{k+1}},\dots,x_{i_n}]$ as its \grobner{} basis.
%
%
	If $I$ is a prime ideal of $\q[X]$ and $X'\subset X$ is non-empty,
then the elimination ideal $I\cap \q[X']$ is prime.
%

%
%
%
%

\subparagraph{Algebraic matroid of a prime ideal.}
\label{subsection:AlgMatPrime} 
Intuitively, a collection of variables $X'$ is \emph{independent} with respect to an ideal $I\subset \q[X]$ if it is not constrained by any polynomial in $I$, and \emph{dependent} otherwise. The \emph{algebraic matroid}  induced by the ideal is, informally, a matroid on the ground set of variables $X$ whose independent sets are subsets of variables that are \emph{not} supported by any polynomial in the ideal.  Its \emph{dependent sets} are supports of polynomials in the ideal.


\subparagraph{Circuits and circuit polynomials.}
\label{sec:circuit polys}
A \emph{circuit} in a matroid is a minimal dependent set. In an algebraic matroid on ground set $X$, a circuit $C\subset X$ is a minimal set of variables supported by a polynomial in the prime ideal $I$ defining the matroid. A polynomial whose support is a circuit $C$ is called a \emph{circuit polynomial} and is denoted by $p_C$. A theorem of Dress and Lovasz  \cite{dressLovasz} states that, up to multiplication by a constant, \emph{a circuit polynomial $p_C$ is the unique polynomial in the ideal with the given support $C\subset X$}. We'll just say, shortly, that it is \emph{unique}. Furthermore, the circuit polynomial is \emph{irreducible}. In summary, \emph{circuit polynomials generate elimination ideals supported on circuits.}


\subparagraph{The Cayley-Menger ideal.}
\label{sec:prelimCMideal}
%
%
When working with the Cayley-Menger ideal we use variables $X_n = \{x_{i,j}: 1\leq i<j\leq n\}$ for unknown squared distances between pairs of points. 
%
The \emph{distance matrix} of $n$ labeled points is the matrix of squared distances between pairs of points. The \emph{Cayley matrix} is the distance matrix bordered by a new row and column of 1's, with zeros on the diagonal. 

\vspace{-14pt}
\begin{center}
$$
\begin{pmatrix}
	0 & 1 & 1 & 1 & \cdots & 1\\
	1 & 0 & x_{1,2} & x_{1,3} & \cdots & x_{1,n}\\
	1 & x_{1,2} & 0 & x_{2,3} & \cdots & x_{2,n}\\
	1 & x_{1,3} & x_{2,3} & 0 & \cdots & x_{3,n}\\
	\vdots & \vdots & \vdots &\vdots &\ddots &\vdots\\
	1 & x_{1,n} & x_{2,n} & x_{3,n} & \cdots & 0
\end{pmatrix}
$$
\end{center}

\noindent
Cayley's Theorem says that, if the distances come from a point set in the Euclidean space $\reals^d$, then the rank of this matrix must be at most $d+2$. Thus all the $(d+3)\times(d+3)$ minors of the Cayley-Menger matrix should be zero.  These minors induce polynomials in $\q[X_n]$ and are the \emph{standard generators} of the $(n,d)$-Cayley-Menger ideal $\cm{n}^d$. They are \emph{homogeneous polynomials} with integer coefficients and are \emph{irreducible} over $\q$. The $(n,d)$-Cayley-Menger ideal is a \emph{prime ideal} of dimension $dn-{\binom{d+1}{2}}$ \cite{borcea:cayleyMengerVariety:2002, Giambelli, HarrisTu, JozefiakLascouxPragacz} and codimension $\binom{n}{2}-dn+{\binom{d+1}{2}}$. In this paper we work with $d=2$ and denote by $\cm{n}$ the $(n,2)$-Cayley-Menger ideal. Its generators are tabulated in \cref{sec:generatorsCM}. 
%
The \emph{algebraic matroid $\amat{\cm{n}^d}$ of the Cayley-Menger ideal} is the matroid on the ground set $X_n$, where a subset of distance variables $X\subseteq X_n$ is independent if $\cm{n}^d\cap~\q[X]=\{0\}$, i.e. $X$ supports no polynomial in the ideal.
%
%
	The rank of $\amat{\cm{n}^d}$ is equal to $\dim{\cm{n}^d}=dn-{\binom{d+1}{2}}$. In particular, the rank of $\cm{n}$ is the Laman number $2n-3$. In fact, the algebraic matroid $\amat{\cm{n}^d}$ is \emph{isomorphic} to the (combinatorial) rigidity matroid introduced in \cref{sec:prelimCombinatorial}. The isomorphism maps Cayley variables $x_{i,j}$ to edges $ij$, thus the \emph{support} of a polynomial in $\cm{n}$ is in one-to-one correspondence with a subgraph of the complete graph $K_n$. All polynomials in $\cm{n}$ are, by definition, dependent and their supports define dependent graphs in the rigidity matroid. The supports of \emph{circuit polynomials} are rigidity circuits.


\subparagraph{Resultants and elimination ideals.} 
\label{sec:prelimResultants}
%
%

%
The resultant\footnote{Historically, it has also been called the \emph{eliminant}.} is the determinant of the Sylvester matrix associated to the coefficients of two uni-variate polynomials, see e.g.\ \cite{CoxLittleOshea}. We skip here the technical details, retaining only the key property that the resultant of two multi-variate polynomials relative to a common variable $x$ is a new polynomial in the union of the variables of the two polynomials, minus $x$. If $R$ is a polynomial ring and $I$ is an ideal in $R[x]$ with $f,g\in I$, then $\res{f}{g}{x}$ is in the elimination ideal $I\cap R$ of the ring $R$. In the context of the Cayley-Menger ideal, the effect (on the supports of the involved polynomials) of taking a resultant is captured by the combinatorial resultant defined in \cref{sec:prelimCombinatorial}.

\section{Standard generators of the 2D Cayley-Menger ideal.}
\label{sec:generatorsCM}

The standard generators of the 2D Cayley-Menger ideal are the $5\times 5$ minors of the Cayley matrix. Each standard generator $g$ is identified with its support graph $G_g$. To motivate the possible choices for the family of graphs {\em Gen} for the {\em combinatorial resultant trees} defined in \cref{sec:extendedCombRes}, we now tabulate the support graphs of all standard generators, up to multiplication by an integer constant, relabeling and graph isomorphism. 

To find all such support graphs, it is sufficient to consider the set of all $5\times 5$ minors of $\cm{10}$. Using a computer algebra package we can verify that this set has 109 619 distinct minors, of which 106 637 have distinct support graphs. 
The IsomorphicGraphQ function of Mathematica was used to reduce them to the $14$ graph isomorphism classes, 11 of which being shown in \cref{fig:generatorGraphs}. The only two representatives with less than 6 vertices are $K_4$ and $K_5$. There are 3 isomorphism classes on 6, 7, 8 vertices (one is $K_6$), 2 on 9 and one on 10 vertices. The corresponding polynomials are, up to isomorphism (relabeling of variables induced by relabeling of the vertices), unique for the given support, with a few exceptions: for $K_5$, we found $3$ distinct (non-isomorphic) polynomials.

\begin{figure}[ht]
	\centering
	\includegraphics[width=0.33\textwidth]{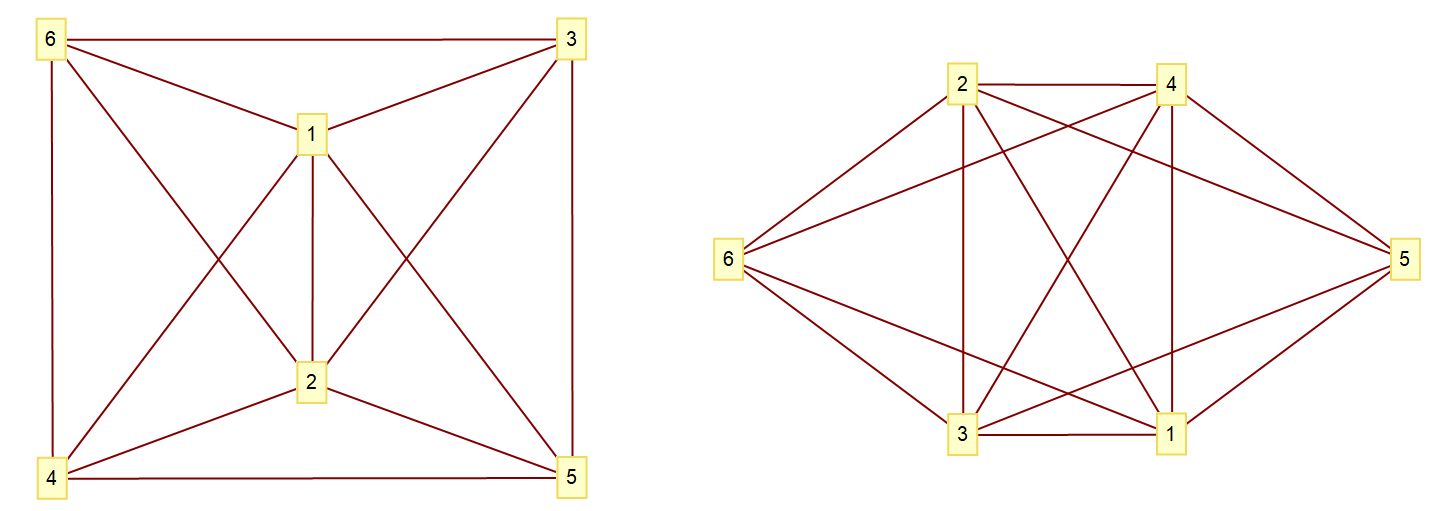}
	\includegraphics[width=0.45\textwidth]{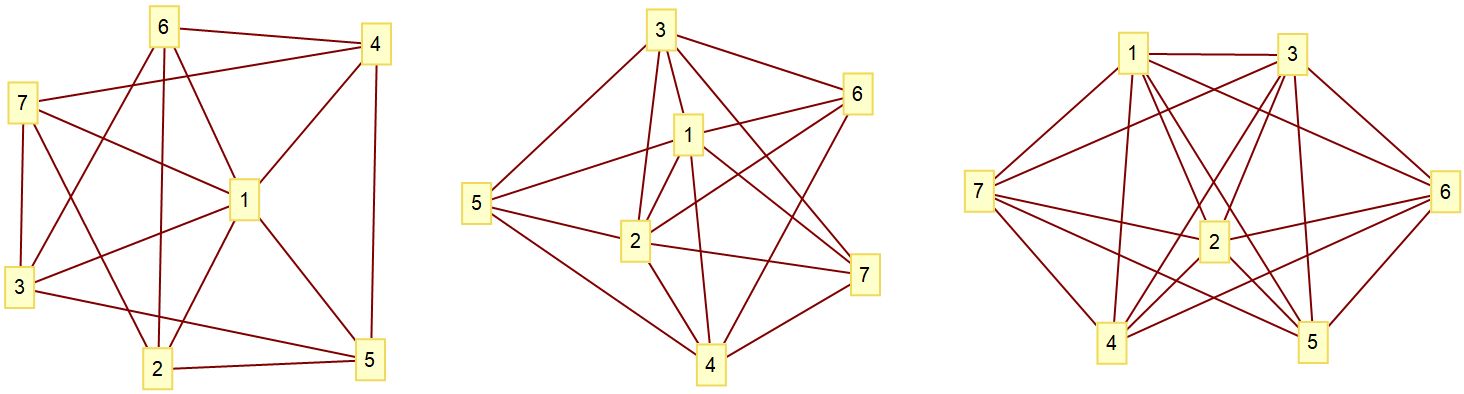}\\
	\includegraphics[width=0.45\textwidth]{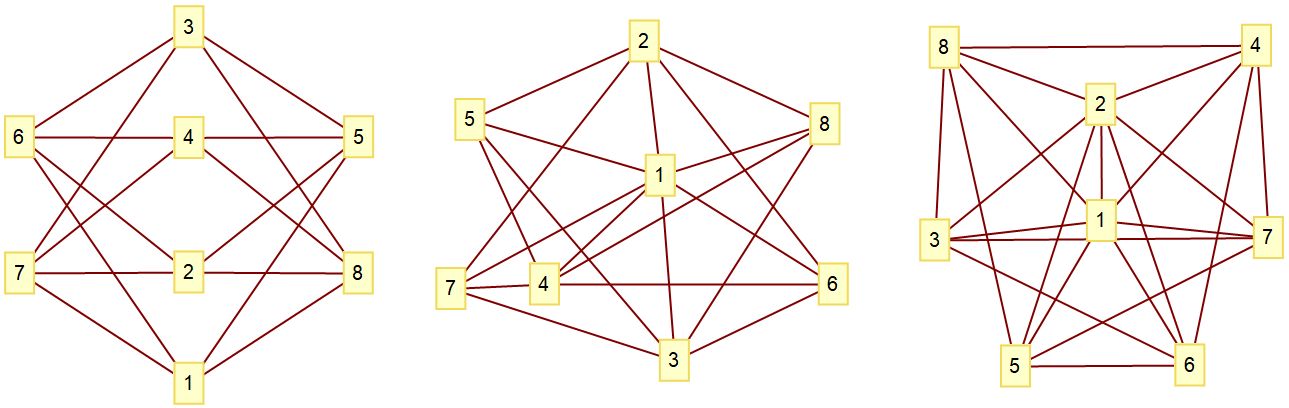}
	\includegraphics[width=0.3\textwidth]{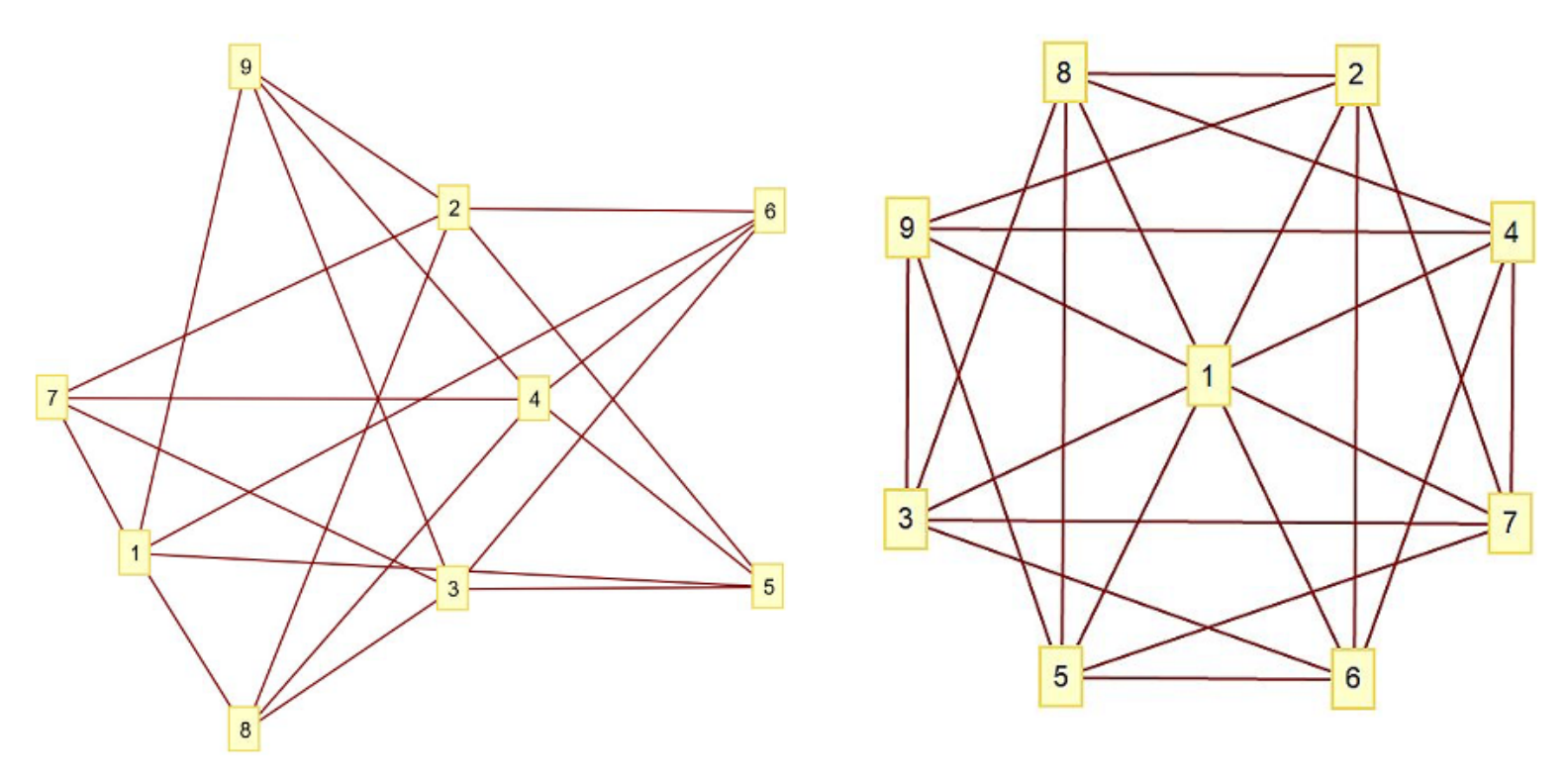}
	\includegraphics[width=.15\textwidth]{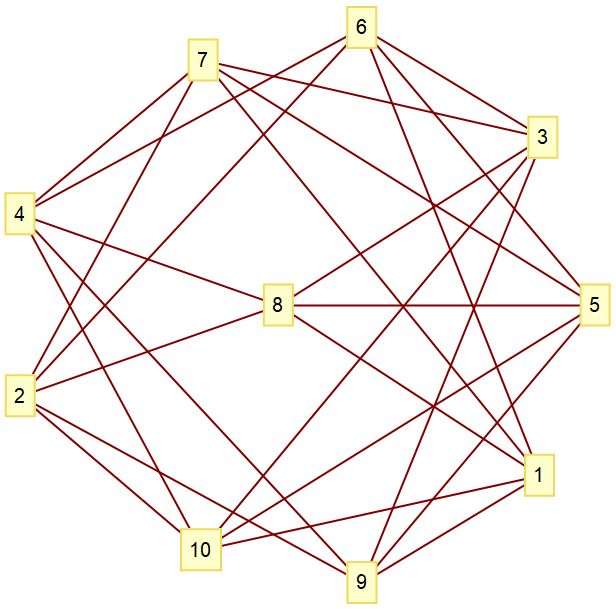}\\
	
	\label{fig:generatorGraphs}
	\caption{The $14$ graph isomorphism classes of Cayley-Menger generators consist in the three complete graphs $K_4, K_5$, $K_6$ and the 11 graphs on $6$ to $10$ vertices shown here.}
\end{figure}

\section{Example: the $K_{3,3}$-plus-one circuit polynomial.}
\label{sec:computingK33}

As a preview of the algorithm for computing circuit polynomials that will be given in \cref{sec:algebraicNew}, we demonstrate now how the algorithm will use the (extended) combinatorial resultant tree from \cref{fig:resTreeK33} to guide the computation of the circuit polynomial for the $K_{3,3}$-plus-one circuit shown at the root. 

At the leaves of the tree we are using irreducible polynomials from among the generators of the Cayley-Menger ideal. The polynomials corresponding to the nodes on the leftmost path from a leaf to the root are refered to, below, as $D_1$ (leftmost leaf), $D_2$ and $D_3$ (for the next two internal nodes with dependent graphs on them) and $C$ for the circuit polynomial at the root. The leaves on the right are three $K_4$ circuit polynomials: $C_1$ supported on vertices $\{ 1,2,3,5 \}$, $C_2$ supported on $\{ 1,3,4,6 \}$ and $C_3$ supported on $\{ 1,4,5,6 \}$. For the polynomial $D_1$ at the bottom leftmost leaf, supported by a dependent $K_5$ graph, we have used the generator:

\begin{align*}
&x_{15} x_{34}^2-x_{16} x_{34}^2-x_{56} x_{34}^2-x_{14} x_{35} x_{34}+x_{16} x_{35} x_{34}+x_{14} x_{36} x_{34}-2 x_{15} x_{36} x_{34}\\
&+x_{16} x_{36} x_{34}-x_{13} x_{45} x_{34}+x_{16} x_{45} x_{34}+x_{36} x_{45} x_{34}+x_{13} x_{46} x_{34}-2 x_{15} x_{46} x_{34}\\
&+x_{16} x_{46} x_{34}+x_{35} x_{46} x_{34}-2 x_{36} x_{46} x_{34}+x_{13} x_{56} x_{34}+x_{14} x_{56} x_{34}-2 x_{16} x_{56} x_{34}\\
&+x_{36} x_{56} x_{34}+x_{46} x_{56} x_{34}-x_{14} x_{36}^2+x_{15} x_{36}^2-x_{13} x_{46}^2+x_{15} x_{46}^2-x_{35} x_{46}^2+x_{14} x_{35} x_{36}\\
&-x_{16} x_{35} x_{36}-x_{36}^2 x_{45}+x_{13} x_{36} x_{45}-2 x_{14} x_{36} x_{45}+x_{16} x_{36} x_{45}-2 x_{13} x_{35} x_{46}+x_{14} x_{35} x_{46}\\
&+x_{16} x_{35} x_{46}+x_{13} x_{36} x_{46}+x_{14} x_{36} x_{46}-2 x_{15} x_{36} x_{46}+x_{35} x_{36} x_{46}+x_{13} x_{45} x_{46}\\
&-x_{16} x_{45} x_{46}+x_{36} x_{45} x_{46}-x_{13} x_{36} x_{56}+x_{14} x_{36} x_{56}+x_{13} x_{46} x_{56}-x_{14} x_{46} x_{56}
\end{align*}

The set of generators supported on $K_5$ contains more than this polynomial. There are two other available choices, of homogeneous degrees 4 or 5, which, in addition, can have quadratic degree in the elimination indeterminate $x_{35}$. The choice of this particular generator was done so as to minimize the complexity of (the computation of) the resultant: its homogeneous degree $3$ and degree $1$ in the elimination variable $x_{35}$ are minimal among the three available options. 

At the internal nodes of the tree we compute, using resultants and factorization, irreducible polynomials in the ideal whose support matches the dependent graphs of the combinatorial tree, as follows. 





The resultant $p_{D_2}=\res{g_{D_1}}{p_{C_1}}{x_{35}}$ is an irreducible polynomial supported on the graph $D_2$ in \cref{fig:resTreeK33}. This graph contains the final result $K_{3,3}$-plus-one as a subgraph, as well as two additional edges, which will have to be eliminated to obtain the final result. {\em Thus the resultant tree is not strictly increasing with respect to the set of vertices along a path, as was the case in \cite{malic:streinu:CombRes:socg:2021}.} However, when the set of vertices remains constant (as demonstrated with this example), the dependent graphs on the path towards the root are {\em strictly decreasing with respect to the edge set.}

The resultant $p_{D_3}=\res{p_{D_2}}{p_{C_2}}{x_{13}}$ is a reducible polynomial with 222108 terms and two non-constant irreducible factors. Only one of the factors is supported on $D_3$, with the other factor being supported on a minimally rigid (hence independent) graph. Thus this factor, the only one which can be in the CM ideal (and it must be, by primality considerations), is chosen as the new polynomial $p_{D_3}$  with which we continue the computation.

The final step to obtain $C$ is to eliminate the edge $46$ from $D_3$ by a combinatorial resultant with $C_3$. The corresponding resultant polynomial $p_{C}$ is a reducible polynomial with 15 197 960 terms and three irreducible factors. As in the previous step, the analysis of the supports of the irreducible factors shows that only one factor is supported on the $K_{3,3}$-plus-one circuit, while the other two factors are supported on minimally rigid graphs. This unique irreducible factor is the desired circuit polynomial for the $K_{3,3}$-plus-one circuit.

The computational time on AMD Ryzen 9 5950x CPU with 64GB of DDR4-3600 RAM in Mathematica v12.3, including reducibility checks and factorizations was 1788.65 seconds. The computation and factorization of the final resultant step took up most of the computational time (1023.7, resp.\ 748.056 seconds). The circuit polynomial for the $K_{3,3}$-plus-one circuit is stored in the compressed WDX Mathematica format (file size approx.\ 7MB) on the GitHub repository \cite{malic:streinu:GitHubRepo}.

\section{Algorithm: circuit polynomial from 
combinatorial resultant tree.}
\label{sec:algebraicNew}

We now have all the ingredients for describing the specificities of the new algorithm for computing a circuit polynomial from a given combinatorial resultant tree for a circuit $C$. Overall, just like the algorithm of \cite{malic:streinu:CombRes:socg:2021}, it computes resultants at each node of the tree, starting with the resultants of generators of $\cm{n}$ supported on leaf nodes. At the root node the circuit polynomial for $C$ is extracted from the irreducible factors of the resultant at the root. The main difference lies at the intermediate (non-root) nodes, as described in \cref{alg:resTreeAlg1} below. This is because the polynomials sought at non-leaf nodes, not being supported on circuits, are not necessarily irreducible polynomials {\em supported on the desired dependent graph} as was the case in \cite{malic:streinu:CombRes:socg:2021}. Hence, conceivably, they may have factors  that are not in the Cayley-Menger ideal, or there may be several choices of factors supported on the dependent graph, or a combination of these cases. It remains, however, as an open question (which may entail experimentation with gigantic polynomials) to explicitly find such examples (we did not find any so far) and to prove what may or may not happen.

\begin{algorithm}[ht]
	\caption{Computing a polynomial in the Cayley-Menger ideal supported on a node of a combinatorial resultant tree}
	\label{alg:resTreeAlg1}
	\textbf{Input}: Non-leaf node $G$ of a combinatorial resultant tree $T_C$. Polynomials $v,w\in\cm{n}$ supported on the child nodes of $G$.\\
	\textbf{Output}: Polynomial $p\in\cm{n}$ supported on $G$.
	\begin{algorithmic}[1]
		\State Compute the resultant $r=\res{v}{w}{x_e}$, where $x_e$ is the indeterminate to be eliminated according to $T_C$.
		\State Factorize $r$ over $\mathbb Q$.
		\If{$r$ has a unique irreducible factor $p$ in $\cm{n}$ supported on $G$}
		\State \Return $p$
		\ElsIf{$r$ factors as $p\cdot q$ such that $p$ is possibly a reducible polynomial in the Cayley-Menger ideal supported on $G$}\label{alg:factorStep}
		\State \Return $p$
		\EndIf
	\end{algorithmic}
\end{algorithm}


\subparagraph{Summary of differences from \cite{malic:streinu:CombRes:socg:2021}.} Since we do not require to have a circuit polynomial at each node, there may be a non-unique choice among the factors of the resultant $r$ that are in the Cayley-Menger ideal {\em and} supported on the rigid graph $G$. When $r$ factors as in step \ref{alg:factorStep}, the following can occur:
\begin{enumerate}
	\item $r$ is irreducible or it has a unique irreducible factor in the CM ideal supported on $G$.
	\item $r$ has 2 or more irreducible factors in the CM ideal supported on $G$. Hence we can choose $p$ to be any of the factors supported on $G$; in principle we choose the factor which will minimize the computational cost in the subsequent step. Heuristically this will be the factor with minimal degree in the indeterminate that is to be eliminated in the subsequent step since this will keep the dimension of the determinant to be computed as small as possible. If there is more than one such factor, we choose one with the least homogeneous degree.
	\item $r$ does not have an irreducible factor in the CM ideal supported on $G$. In this case, for $p$ we have to choose a reducible factor of $r$ that is in the CM ideal and supported on $G$. Such a reducible factor is not necessarily unique. If it is not unique, we choose one according to the same heuristics as in the previous case.
\end{enumerate}

\section{Concluding Remarks.}
\label{sec:concludingRemarks}

We have extended the algorithmic approach of \cite{malic:streinu:CombRes:socg:2021} for computing circuit polynomials. The additional structure in the Cayley-Menger ideal that we identified allows for a systematic calculation of polynomials that were previously unatainable with the only generally available method, the double-exponential Gr\"obner Basis algorithm, or with the previously, recently proposed method of \cite{malic:streinu:CombRes:socg:2021}. The new perspective on distance geometry problems offered by combinatorial resultants raises further directions of research and open questions, besides those stated by \cite{malic:streinu:CombRes:socg:2021}, whose answers may clarify the theoretical complexity of the algorithm.

We conclude the paper with a few such open problems concerning the combinatorial and algebraic structure of the (combinatorial and algebraic) resultant in the context of the Cayley-Menger ideal.

\begin{problem}{
	We work now with an extended collection of generators, not all of them circuits (such as those from \cref{sec:generatorsCM}). Given a circuit, decide if it has a combinatorial resultant tree with at least one non-$K_4$ leaf from the given generators.}
%
\end{problem}

\begin{problem} Consider an intermediate node $G$ in a combinatorial resultant tree and let $r=\res{v}{w}{x_e}$ be the resultant supported on $G$ with respect to the polynomials supported on the child nodes of $G$, as in Algorithm \ref{alg:resTreeAlg1}. Find examples where $r$ factors as $r=f\cdot g$, such that neither $f$ nor $g$ are supported on $G$ and are not necessarily irreducible, or prove that this never happens.
\end{problem}


%
%
%

\bibliography{references}

\end{document}